\documentclass[11pt,twoside,a4paper,notitlepage]{article}

\usepackage{amsmath}

\usepackage{amssymb}
\usepackage{amsfonts}
\usepackage[a4paper,margin=2cm,vmargin={1cm,2cm},includeheadfoot]{geometry}
\usepackage[T1]{fontenc}
\usepackage{amsthm}
\usepackage{enumerate}
\usepackage[scaled=0.92]{helvet}
\usepackage{srcltx}
\usepackage{graphicx}
\usepackage{tikz}
\usepackage{color}

\usepackage[labelfont=bf,margin=1cm,justification=justified,labelsep=period]{caption}

\theoremstyle{plain}
\newtheorem{theorem}{Theorem}
\newtheorem{proposition}[theorem]{Proposition}

\newtheorem{lemma}[theorem]{Lemma}
\newtheorem{corollary}[theorem]{Corollary}

{\theoremstyle{definition}
\newtheorem{remark}[theorem]{Remark}}

{\theoremstyle{definition}
\newtheorem{remarks}[theorem]{Remarks}}

\newcommand{\R}{{\mathbb {R}}}

\newcommand{\ul}{\underline}
\newcommand{\ol}{\overline}

\begin{document}

\begin{center}
{\LARGE\bf
Approximating Mills ratio} \\ [0.5 cm]

{\sc\Large Armengol Gasull and Frederic Utzet} \\ [0.5 cm]
\end{center}

\noindent{\sl\large Departament de Matem\`{a}tiques, Edifici C,  Universitat Aut\`{o}noma de Barcelona,
08193 Bellaterra (Barcelona) Spain.  E-mails:  gasull@mat.uab.cat,
 utzet@mat.uab.cat}


\begin{center}
\begin{minipage}{0.9\linewidth}
\begin{small}
\noindent{\bf Abstract.}

\smallskip

Consider the Mills ratio  $f(x)=\big(1-\Phi(x)\big)/\phi(x), \, x\ge 0$,  where $\phi$ is the
density  function of the standard Gaussian law and $\Phi$ its cumulative distribution.
We introduce a general procedure to approximate $f$ on the
whole $[0,\infty)$  which
   allows to prove  interesting  properties where
  $f$ is involved.
As  applications  we  present a new proof
that $1/f$ is strictly convex, and we give new sharp  bounds of $f$
involving rational functions, functions with  square roots or exponential terms.
Also Chernoff type bounds for the Gaussian $Q$--function are studied.

\smallskip

\smallskip

\end{small}
\end{minipage}
\end{center}
\section{Introduction}

Recall that the Mills ratio
 (Mills \cite{Mills}) is the function
\begin{equation}\label{1} f(x)=\frac{1-\Phi(x)}{\phi(x)}={e}^{\frac{x^2}2}\int_x^\infty
{e}^{-\frac{t^2}2}\,dt, \ x\ge 0,
\end{equation}
 where
$\phi(x)=e^{-x^2/2}/\sqrt{2\pi}$ is the density function of a
standard Gaussian  law and $\Phi(x)=\int_{-\infty}^x\phi(t)\, dt$
its cumulative distribution function. The study of this function is much older than
Mills \cite{Mills}, and through   its relation with the function
\begin{equation}
\label{functionLaplace}
F(x)=e^{x^2}\int_x^\infty e^{-t^2}\, dt
\end{equation}
given by
$$f(x)=\sqrt 2\, F(x/\sqrt 2),$$
its introduction
can be traced back to Laplace (1805) \cite[Livre X, Chap.
1, n$^{\rm o}$ 5]{laplace}, while he was analyzing different
hypotheses related with the refraction of the light in the
atmosphere. As we will  comment later,  Laplace gave  many of the essential results, like the
continued fraction and   the asymptotic expansion. Moreover,
 since  the function $F$ is related with the error function,
 and also with the upper incomplete Gamma function  of parameter $1/2$,
properties of Mills ratio are spread  between papers and books
of Probability and Statistics, Mathematical Analysis, Numerical Analysis, etc,
and many results have been discovered and rediscovered by different authors.

In the first part of this paper we collect, in a short,  unified and self-contained
way, some known results about the approximation
of $f$ by rational funcions. In particular, we prove the surprising fact that the convergents of the continued fraction
of $f$ give, at the same time, the expression of its $n^{\rm th}$-derivative.

In a second part, joining the rational  bounds  for $f(x)$ when
$x$ is large, with a Taylor formula for small $x$, we construct
effective bounds for  $f$ on the whole $[0,\infty)$.
From that, we introduce  a general procedure to prove properties
where $f$ is involved; this procedure consists
on the reduction to the problem at hand  to the control of the real roots of
a polynomial with rational coefficients, which  can be done using rigorous  analytic methods
based on the Sturm Theorem (see the Appendix for more details on that theorem).

As a first application we
prove that the reciprocal function of  Mills ratio  $1/f$ is strictly convex.
This result  was implicitly conjectured by Birnbaum \cite{Bir50} in 1950
(see Subsection \ref{Sub:Main})
and  demonstrated few years later  independently by  Sampford \cite{Sam} and
 Shenton \cite{se}.
We stress that, in contrast with  known   proofs,
we apply a general procedure which, as we will see, can be useful in many other problems.

As a second
application we study  lower and upper bounds for $f$
of the form
$$\psi_{a,b,c}(x)=\frac{a}{\sqrt{x^2+b}+cx}.$$
These functions have the same type of  asymptotic expansion as $f$ when $x\to \infty$,
and then very sharp bounds for $f$ can be obtained.
These bounds are important because they give  very good estimations of the Gaussian
 cumulative
probability distribution. We review in  a  systematic way known bounds of this
type, and construct other new.

 Finally, we shortly present new families of simple bounds of $f$ of the form $a/(b+x)$,
 $(a+bx)/(c+dx+x^2)$ and $\big(1-\text{exp}(-ax)\big)/bx$. We also comment on  Chernoff type bounds
 $a\, \text{exp}(-bx^2)$ for the function $\int_x^\infty \text{exp}(-t^2/2)\, dt/\sqrt{2\pi}$, called
 the Gaussian $Q$--function in engineering literature.  All  these
  properties can  be proved  by using our methodology.

\section{Approximation of Mills ratio by rational functions}

\subsection{Notations and main results}

Starting from $f'(x)=xf(x)-1$, it follows that the $n^{\rm th}$-derivative of $f$,
$n\geq1,$ satisfies
\begin{equation}\label{rec-der2}
 f^{(n)}(x)=P_n(x) f(x)-Q_n(x),
\end{equation} for some polynomials  $P_n$ and $Q_n$, with non negative integer coefficients and respective  degrees
$n$ and $n-1$. We prove:

\begin{theorem}\label{main} Let $P_n$ and $Q_n$ be the polynomials defined by~\eqref{rec-der2}
where $f$ is the Mills ratio~\eqref{1}. Then, for all $n\ge1$ and
$x\ge0,$
\begin{equation}\label{fita}0<(-1)^n
\frac{f^{(n)}(x)}{P_n(x)}=(-1)^n\left(f(x)-\frac{Q_{n}(x)}{P_{n}(x)}\right)<
\frac{n!}{x^{2n+1}}.\end{equation} Moreover, for all $x>0$,
 when ${m\to \infty}$,
$$\frac{Q_{2m}(x)}{P_{2m}(x)}\nearrow f(x) \quad
\text{and}\quad \frac{Q_{2m+1}(x)}{P_{2m+1}(x)}\searrow f(x).$$
\end{theorem}

Our proof is similar to  the approach
presented in the unpublished (as far as we know) work  of Kouba~(\cite{K})
 that we discovered after preparing  a first
version of this paper.

 As a
corollary of Theorem~\ref{main} we obtain:

\begin{corollary}\label{coromain}

\rule[0mm]{0mm}{0mm}

\begin{enumerate}[(i)]

\item The Mills ratio $f$ is completely monotone.

\item The rational functions $Q_n/P_n$  are the  convergents of the continued
 fraction expansion of $f.$

\item The rational functions $Q_n/P_n$ are the Pad\'{e}--Laurent
approximants of $f$ at infinity.

\item The polynomial $P_n(x)$
  coincides with  $H\!e_n(ix)/i^n$, where $i^2=-1$, and $H\!e_n(x)$ is
  the monic Hermite polynomial of order $n$  with respect to the weight function $e^{-x^2/2}$.

\end{enumerate}

\end{corollary}

Item {\it (i)} is already known, see the comments below. The convergents
and the Pad\'{e}--Laurent approximants of $f$ are also known, but they
are obtained using a different method. We obtain them  simply using
equality \eqref{rec-der2}. Let us comment with more detail all the
items of the corollary.

Recall that a function $h:(0,\infty)\to \R$ is called {\it completely
monotone} if it is of class ${\cal C}^\infty$ and for all $n\ge 0$,
$$(-1)^nh^{(n)}(x)\ge 0,\  x>0,$$
where $h^{(0)}=h$.
 Completely monotone functions  and the related Bernstein functions
constitute   a topic of permanent interest due to its apparition in
several different areas of Mathematics and its many applications;
for a complete treatise  see Schilling {\it et al.}
\cite{schilling12}.

The result stated in item {\it (i)} was  proved by
Baricz~\cite{bar}. Its proof is based on the following nice equality
\begin{equation}\label{lap}
f(x)={e}^{\frac{x^2}2}\int_x^\infty
{e}^{-\frac{t^2}2}\,dt=\int_0^\infty {e}^{-{x
t}}{e}^{-\frac{t^2}2}\,dt,
\end{equation}
that already appears in Ray and Pitman \cite[Formula (9)]{RayPit63}. From this equality the sign of the $n^{\rm th}$ derivative is simply
obtained using the right-hand side expression of~$f$. Recall that
Bernstein Theorem (see Schilling {\it et al.} \cite[Theorem
1.4]{schilling12}) characterizes completely monotone functions as
the ones that are the Laplace transform of a (univocally
determinate) measure on $[0,\infty)$. Equality \eqref{lap} gives
explicitly this measure.

In item {\it (ii)} we prove that
 $Q_n/P_n$ are the convergents of the continued fraction expansion of $f$:
\begin{equation}
\label{fracciocont} f(x)=\cfrac{1}{x+\cfrac{1}{x+
\cfrac{2}{x+\cfrac{3}{x+ \cdots}}}}\,.
\end{equation}
The expansion corresponding to the function $F$ given in (\ref{functionLaplace})
 was obtained by Laplace  \cite[Livre X, Chap.
1, n$^{\rm o}$ 5]{laplace}
and  the continued
fraction (\ref{fracciocont}) can be easily deduced from  Laplace's
one. For a direct proof of (\ref{fracciocont}) see Small
\cite[Section 3.5]{small10}.

One of the main steps in the proof of Theorem \ref{main} and of item
{\it (ii)} is to show that $\{P_n(x)\}_n$ and $\{Q_n(x)\}_n$ satisfy some
second order recurrences, see Lemma~\ref{l3}. We obtain these
recurrences by using that the differential equation $y'(x)-xy(x)=1$ has
no rational solutions. It is interesting to comment that this is one
of the key points in the celebrated Liouville's proof that the
distribution function of the Gaussian law has no primitive which can
be expressed in terms of elementary functions, see Remark \ref{Liouville}.

In item {\it (iii)} we show that   the rational functions $Q_n/P_n$ are
some Pad\'{e}--Laurent approximants of $f$ at infinity. Similar results
 appear in \cite{p}. In fact the asymptotic expansion of $f$ at
infinity is (see Small \cite[p. 44]{small10})
\begin{equation}
\label{asy_feller}
f(x)\ \sim\   \frac{1}{x}-\frac{1}{x^3}+\frac{1\cdot 3}{x^5}-
\frac{1\cdot 3\cdot 5}{x^7}+\cdots,  \ x\to\infty.
\end{equation}
This expansion also comes  from Laplace  \cite[Livre X, Chap. 1, n$^{\rm o}$ 5]{laplace}
(with the same comments as above with respect to the function involved).
The bounds for $f$ deduced from (\ref{asy_feller}) are widely used in Probability and
Statistics; for example, Hall \cite{Hall79} uses  the first two terms as a key ingredient
 to get the rate of convergence in the supremum metric of the maxima of standard normal random variables
to Gumbel law. Denote by  $J_n(x)$ the finite expansion of
$f(x)$ up to $1/x^{2n-1}$ deduced from (\ref{asy_feller}):
$J_1(x)={1}/{x},$ and
\begin{equation}\label{jj}
J_n(x)= \frac{1}{x}-\frac{1}{x^3}+\cdots +(-1)^{n+1} \frac{
(2n-3)!!}{x^{2n-1}},\ \ \text{for $n\ge 2$},\end{equation}
It is well known (see Small \cite[Section 2.3]{small10}) that
the error term of the finite expansion
 is bounded by the first neglected term, that is,
\begin{equation}\label{ineq}0<(-1)^n\left(f(x)-J_n(x)\right)< \frac{
(2n-1)!!}{x^{2n+1}}.\end{equation}
Since $n!<(2n-1)!!$,  the functions $Q_n(x)/P_n(x)$ seem to approach
faster $f$  than the functions $J_n(x)$ (see Theorem \ref{main}). This turns out to be true,
see Lemma~\ref{ppp}. A possible explanation of this fact comes from
item {\it (iii)} of Corollary~\ref{coromain}. As often happens   the Pad\'{e}
approximants give better approximation of the function than
truncating  the corresponding  Taylor series.

\bigskip

We thank  Iain Johnstone for pointing out the property given in item
{\it (iv)}. It was missed in a first version of this work. From the explicit expression
of the Hermite polynomials NIST \cite[Formulas 18.5.13 and 18.7.12]{NIST10}
it is deduced that
\begin{equation}
\label{PnExplicit}
P_n(x)=n!\sum_{k=0}^{[n/2]}\frac{x^{n-2k}}{2^k k!(n-2k)!}.
\end{equation}
Property {\it (iv)} and expression (\ref{PnExplicit})
appear in  Kouba \cite{K}, where also an explicit expression of $Q_n$ is given.

\subsection{Proof of Theorem \ref{main} and Corollary~\ref{coromain}}

We first compile several  properties of the
polynomials introduced in~\eqref{rec-der2}, see also Kouba~\cite{K}.

\begin{lemma}\label{l3} Let $P_n$ and $Q_n$ be the polynomials defined
by~\eqref{rec-der2}.
\begin{enumerate}[(a)]
\item It holds that \begin{equation}
\label{recP} P_n(x)=x P_{n-1}(x)+(n-1)P_{n-2}(x),\  n\ge 2,
\end{equation}
with initial conditions $P_0(x)=1$ and $P_1(x)=x,$ and
\begin{equation}
\label{recQ} Q_n(x)=x Q_{n-1}(x)+(n-1)Q_{n-2}(x),\  n\ge 2,
\end{equation}
with initial conditions $Q_0(x)=0$ and $Q_1(x)=1.$ In particular,
both polynomials  $P_n$ and $Q_n$ are monic and  their coefficients
are nonnegative integers.

\item For $n\ge 1$,
\begin{equation}
\label{derivada} P_{n}'(x)=nP_{n-1}(x) \qquad\mbox{and}\qquad
Q'_n(x)=xQ_n(x)+nQ_{n-1}(x)-P_n(x).
\end{equation}

\item For $n\ge1,$
\begin{equation}
\label{prodPQ} Q_n(x) P_{n-1}(x)-Q_{n-1}(x)P_n(x) =(-1)^{n+1}(n-1)!.
\end{equation}
\end{enumerate}
\end{lemma}

\begin{proof} $(a)$ From~\eqref{1} and the equalities $\Phi'(x)=\phi(x)$ and
$\phi'(x)=-x\phi(x),$ we arrive to
\begin{equation}
\label{aprima} f'(x)=xf(x)-1\qquad\mbox{and}\qquad
f''(x)=(x^2+1)f(x)-x.
\end{equation}
In general, for $n\ge 2$, thanks to Leibnitz formula for the
derivative of order $n$ of a product of functions, we have the
recurrence
\begin{equation}
\label{rec-der1} f^{(n)}(x)=xf^{(n-1)}(x)+(n-1) f^{(n-2)}(x),
\end{equation}
where $f^{(0)}=f.$ Given the form of $f'$ and $f''$ in
(\ref{aprima}), it is deduced by induction that~\eqref{rec-der2}
holds, where $P_n$ and $Q_n$ are polynomials of degree $n$ and $n-1$
respectively (for $n\ge 1$). For $n=0$ the corresponding polynomials
are $P_0(x)=1$ and $Q_0(x)=0$.
The first polynomials are
$$ P_1(x)=x,\ P_2(x)=x^2+1, \ P_3(x)=x^3+3 x, \ P_4(x)=x^4+6x^2+3,\ P_5(x)=x^5+10x^3+15x,$$
and
$$ Q_1(x)=1,\ Q_2(x)=x,\ Q_3(x)=x^2+2, \ Q_4(x)=x^3+5x, \ Q_5(x)=x^4+9x^2+8.$$
Note that from (\ref{rec-der2}) and (\ref{rec-der1}),
\begin{equation}
\label{igualtatpol}
\big(P_n(x)-xP_{n-1}(x)-(n-1)P_{n-2}(x)\big)f(x)=Q_n(x)-xQ_{n-1}(x)-
(n-1)Q_{n-2}(x).
\end{equation}
For any $n$ given, the factor $P_n(x)-xP_{n-1}(x)-(n-1)P_{n-2}(x)$
is a polynomial of degree $n$, and thus, if it is not identically
zero we get that $f$ is a rational function. On the other hand
remember that by \eqref{aprima},  $f'(x)=xf(x)-1$. It is not
difficult to prove that  this equation has no rational solutions:
$f$ cannot be a polynomial by degree considerations. Hence $f$
should have a (real or complex) pole. This pole is also present in
$f'$ but with a different order, giving again the impossibility of
$f$ to be a solution of the equation.
 Then, the polynomials $P_n$
and $Q_n$ follow the recurrences~\eqref{recP} and~\eqref{recQ} as we
wanted to prove.

$(b)$ Notice that
\begin{align*}
\label{derivadesPQ} P_{n+1}(x)f(x)-Q_{n+1}(x)&=f^{(n+1)}(x)=\big(P_n
(x) f(x)-Q_n(x)\big)'\notag \\&
=P'_n(x) f(x)+P_n(x)f'(x)-Q'_n(x)  \notag\\
&=\big(P'_n(x)+xP_n(x)\big) f(x)-\big(P_n(x)+Q'_n(x)\big).
\end{align*}
Then,
\begin{equation}
\label{derPNmes1} P_{n+1}(x)=P'_n(x)+xP_n(x) \qquad\mbox{and}\qquad
Q_{n+1}(x)=P_{n}(x)+Q'_{n}(x).
\end{equation}
Joining the  above left equality with the recurrence (\ref{recP})
for $P_{n+1}$ it follows the first formula in (\ref{derivada}).
Similarly, combining the right one  with the recurrence for
$Q_{n+1}$  in
 (\ref{recQ}), gives the second equality in~\eqref{derivada}.

$(c)$ Equality~\eqref{prodPQ} follows easily by induction, changing
$P_n$ and $Q_n$ by their corresponding  expressions given in
(\ref{recP}) and (\ref{recQ}), respectively.
\end{proof}

\begin{remark}
\label{Liouville} Our proof of item $(a)$ uses one of the steps of the proof that
 $\int e^{-x^2/2}\,dx$
has no primitive expressible in terms of elementary functions. To
realize  this fact
 it is  convenient  a  short comment on this proof. It   goes back to the work
of Liouville 1835 (\cite{l}) and is a consequence of a much more
general result. Liouville Theorem together with its proof can also
be consulted for instance in Rosenlicht \cite{r}. In Conrad
\cite{C}, there is a simple and nice corollary of Liouville's
result: {\it  Consider a function $F(x)=p(x)e^{q(x)}$, where both
$p$ and $q$ are rational functions with $p$ not identically zero and
$q$ not constant. Then $F$ can be integrated in elementary terms if
and only if there exists a rational function $R$ such that
\begin{equation}\label{cns} R'(x)+q'(x)R(x)=p(x).
\end{equation}
Moreover, in this case $\int p(x) e^{q(x)}dx=R(x) e^{q(x)}+k$ for
some constant $k$.} Notice that for the integral  $\int
e^{-x^2/2}\,dx$, the differential equation \eqref{cns} is precisely
the differential equation \eqref{aprima} satisfied by $f$, which, as
we have already argued in the proof of the above lemma, has no
rational solutions.
\end{remark}


\begin{proof}[Proof of Theorem~\ref{main}] We start proving the inequality
\begin{equation}\label{dn}
(-1)^n
\frac{f^{(n)}(x)}{P_n(x)}=(-1)^n\left(f(x)-\frac{Q_n(x)}{P_n(x)}
\right)>0.
\end{equation}
When $n$ is even  we have to see that (for $n$ odd, reverse the
inequality),  $f(x)>{Q_n(x)}/{P_n(x)},$
 or equivalently,
\begin{equation}
\label{desigualtatPhi} 1-\Phi(x)>\phi(x)\frac{Q_n(x)}{P_n(x)}.
\end{equation}
This inequality is  proved with the clever   argument that Feller
(\cite[page 175]{fellerVol1}) uses to   deduce  the asymptotic
 expansion (\ref{asy_feller}):
We will see that the negative of the derivatives of  those of
(\ref{desigualtatPhi})
 satisfy the inequality
\begin{equation}
\label{desigualt2}
\phi(x)>-\bigg(\phi(x)\frac{Q_n(x)}{P_n(x)}\bigg)'
\end{equation}
(reversed inequality if $n$ is odd).  Hence, integrating both
members of that inequality from $x$ to infinity we obtain
(\ref{desigualtatPhi}).

To prove (\ref{desigualt2}) we first claim that for every $n\ge 1$,
\begin{equation}
\label{pquadrat}
xP_n(x)Q_n(x)-Q'_n(x)P_n(x)+Q_n(x)P'_n(x)=P_n^2(x)+(-1)^{n+1}n!.
\end{equation}
Using this claim we get that the right  hand side of
(\ref{desigualt2}) is
\begin{multline}
\label{derivadaQP}
x\phi(x)\frac{Q_n(x)}{P_n(x)}-\phi(x)\frac{Q'_n(x)P_n(x)-Q_n(x)P'_n(x)}{P_n^2(x)}
\notag\\=\phi(x)\Bigg(\frac{xP_n(x)Q_n(x)-Q'_n(x)P_n(x)+Q_n(x)P'_n(x)}{P_n^2(x)}\Bigg)=\phi(x)\left(1+\frac{(-1)^{n+1}n!}{P_n^2(x)}\right).
\end{multline}
Then,  (\ref{desigualt2}) for $n$ even, and the reversed inequality
for $n$ odd, will  follow.

To prove the claim, in (\ref{pquadrat}), we change $P'_n$ and $Q'_n$
by their corresponding  expressions in~(\ref{derivada}). It turns
out that to prove (\ref{pquadrat}) is equivalent to
prove~\eqref{prodPQ}. Hence, from (c) in Lemma~\ref{l3},  the proof
of~\eqref{dn}  is complete.

Let us prove  the right-hand inequality of~\eqref{fita}. For $n$
even, thanks to identity (\ref{prodPQ}),
\begin{equation}\label{diff}\frac{Q_{n+1}(x)}{P_{n+1}(x)}-\frac{Q_{n}(x)}{P_{n}(x)}=
\frac{n!}{P_{n+1}(x)P_n(x)}.\end{equation}
 Recall that the
polynomials $P_n$ are monic and all their coefficients are
nonnegative. Then $P_{n+1}(x)P_n(x)>x^{2n+1}.$ Thus,
$$0<\frac{Q_{n+1}(x)}{P_{n+1}(x)}-\frac{Q_{n}(x)}{P_{n}(x)}<\frac{n!}{x^{2n+1}},$$
and the result follows  because
\begin{equation}\label{ff}
\frac{Q_{2n}(x)}{P_{2n}(x)}<f(x)<\frac{Q_{2n+1}(x)}{P_{2n+1}(x)}.\end{equation}
Finally, to prove that when $x>0$,
\[
\lim_{n\to\infty}\frac{P_n(x)}{Q_n(x)}=f(x),
\] we will need a suitable lower bound for $P_{n+1}(x)P_n(x).$ We
will consider the independent term and the coefficient of $x$ of
$P_n(x)$. By  (\ref{PnExplicit}),
 $P_n(0)=(n-1)!!$, if $n$ is even,   and 0 otherwise,
where $k!!$ is the double factorial of a positive integer, defined
recurrently for $k\ge2$ as $k!!=k\times (k-2)!!,$ with  the
conventions $0!!=1!!=1$. Also,
$P_n[1]=
n!!$ if $n$ is odd, and 0 otherwise.

 Using again that all coefficient of $P_n$ are non negative, and
$P_{n+1}(0)=0$, for every $x>0$, we can bound the product
$P_{n+1}(x)P_n(x)$ in the following way:
$$P_{n+1}(x)P_n(x)>x P_{n+1}[1]P_n(0)= x\,(n+1)!! (n-1)!!.$$
Thus, for $n$ even, equality~\eqref{diff} gives
$$0<\frac{Q_{n+1}(x)}{P_{n+1}(x)}-\frac{Q_{n}(x)}{P_{n}(x)}<\frac{n!}{x (n+1)!! (n-1)!!}.$$
Fixed $n$, the right hand side goes to 0 when $x\to \infty$, and
fixed $x>0$,   by Stirling formula, it goes to 0 when $n\to \infty$.
 For any $x>0,$ the
monotonicity of the even and odd terms of the sequence
$Q_n(x)/P_n(x)$ is an straightforward consequence of item~(c) of
Lemma~\ref{l3}. Then the theorem follows.
\end{proof}

\begin{proof}[Proof of Corollary~\ref{coromain}]

{\it (i)} Notice that because $P_n(x)>0$ for all $x>0$ the fact that
$(-1)^nf^{(n)}(x)>0$ for all $x>0$ is a straightforward consequence
of~\eqref{fita}.

{\it (ii)} Consider the continued fraction (\ref{fracciocont}) of $f$.
 The {\it three terms recurrence
relation}
 that follow the $n^{\rm th}$ numerator
and the $n^{\rm th}$ denominator of a continued fraction (see Cuyt
{\it et al.} \cite[page 13]{cuyt})
  are exactly the recurrences for $P_n$
 and $Q_n$  given by (\ref{recP}) and (\ref{recQ}). So $Q_n/P_n$ is the $n$ convergent of
 (\ref{fracciocont}). Then,  the properties of  continued fractions can be used to
 give an alternative  proof of items (b) and (c) of Lemma~\ref{l3};
 see  Pinelis~\cite{p} and Shenton~\cite{se}. In fact, there is a
 natural procedure for finding a continued fraction expansion for
 a function satisfying a first order differential equation with
 polynomials entries that goes back to Laguerre, see~\cite{lag}.
 Notice that our Mills ratio $f$ belongs to this class of functions.

{\it (iii)} Given a generic function $S$ we define $\widetilde S(y)=
S(1/y).$ Then, from~\eqref{asy_feller}, \eqref{jj} and \eqref{ineq}
it holds that
\begin{equation}\label{yyy}
\widetilde f(y)=\widetilde J_n(y)+O(y^{2n+1}),
\end{equation}
where $\widetilde J_n(y)= y-y^3+\cdots +(-1)^{n+1}
(2n-3)!!y^{2n-1},$ and, as usual, we write that $k(x)=O(h(x))$ when $x\to \infty$ if there is
a point  $x_0$ and a constant $C$ such that $\vert k(x)\vert \le Ch(x)$, for all $x>x_0$. Consider now the rational functions
\[
\frac{\widetilde Q_{2m}(y)}{\widetilde
P_{2m}(y)}=y\,\frac{R_{m-1}(y^2)}{S_{m}(y^2)},\qquad\frac{\widetilde
Q_{2m+1}(y)}{\widetilde P_{2m+1}(y)}=y\, \frac{T_{m}(y^2)}{U_{m}(y^2)},
\]
where $R_j,\,S_j,\, T_j$ and $U_j$ are polynomials of degree $j$. Notice
that by Theorem~\ref{main}, $\widetilde Q_n(y)/\widetilde
P_n(y)-\widetilde f(y)=O(y^{2n+1})$. Hence, by~\eqref{yyy},
\[
\widetilde J_n(y)- \frac{\widetilde Q_{n}(y)}{\widetilde
P_{n}(y)}=O(y^{2n+1}).\] This equality implies that  $\widetilde
Q_n(y)/\widetilde P_n(y)$ are the Pad\'{e} approximants of $\widetilde
J_n(y)$ and $\widetilde f(y)$ at the origin, of order $[n-1,n]$ when
$n$ is even and of order $[n,n-1]$ when $n$ is odd because all the
derivatives until order $2n$ at the origin coincide. Hence the
corresponding $Q_n(x)/P_n(x)$ are the Pad\'{e}--Laurent approximants of
$f(x)$ at infinity. Finally, notice that by symmetry arguments, when
$n$ is even the Pad\'{e} approximants of orders $[n-1,n], [n-1,n+1],
[n,n]$ and $[n,n+1]$ coincide. Similarly, when $n$ is odd the
approximants that conicide are the ones having orders $[n,n-1],
[n,n]$ and $[n+1,n]$.

 {\it (iv)}  The recurrence of the monic Hermite polynomials with respect to
 the weight function $e^{-x^2/2}$ is (see NIST \cite[Table 18.9(ii)]{NIST10})
 $$H\!e_{n}(x)=xH\!e_{n-1}(x)-(n-1)H\!e_{n-2}(x),$$
 with initial conditions $H\!e_0(x)=1$ and $H\!e_1(x)=x.$
 Comparing with the recurrence of $P_n$ given in (\ref{recP}) the property
 is easily proved.
\end{proof}

\subsection{A comparison between  $J_{n}$ and $Q_n/P_n$}
\label{Subsec:Comparison}

Notice that by~\eqref{jj},
\begin{equation}\label{dj}
J_{n+1}(x)=J_n(x)+(-1)^{n+2} \frac{
(2n-1)!!}{x^{2n+1}}\quad\mbox{and}\quad J'_{n}(x)=xJ_{n+1}(x)-1.
\end{equation}
Next result shows that $J_n$ approximates worst $f$ that $Q_n/P_n$.
This fact corroborates once more the general believe that Pad\'{e}
approximants are better than the corresponding truncated  series.

\begin{lemma}\label{ppp}
For $x>0$ and $m\ge 1$,
$$J_{2m}(x)<\frac{Q_{2m}(x)}{P_{2m}(x)}<f(x)<\frac{Q_{2m+1}(x)}{P_{2m+1}(x)}
<J_{2m+1}(x).$$

\end{lemma}

\begin{proof} As in the proof of Theorem~\ref{main}, we use again  the  idea
 from Feller (\cite[page 175]{fellerVol1}):
 We  prove that for $n$ even
(the proof for $n$ odd is analogous),
$$\phi(x)J_n(x)<\phi(x)\frac{Q_n(x)}{P_n(x)}.$$
To this end, it suffices to check the inequality with the negatives
of the derivatives of both sides of  previous inequality,
\begin{equation}
\label{JQP}
-\phi(x)J'_n(x)+x\phi(x)J_n(x)<-\phi(x)\bigg(\frac{Q_n(x)}{P_n(x)}\bigg)'
+x\phi(x)\frac{Q_n(x)}{P_n(x)},
\end{equation}
and then  integrate from $x$ to infinite. By \eqref{dj} the left
hand side of \eqref{JQP} is
\[
-\phi(x)J'_n(x)+x\phi(x)J_n(x)=\phi(x)\left(1+x\big(J_n(x)-J_{n+1}(x)
\big)\right)=\phi(x)\left(1-\frac{(2n-1)!!}{x^{2n}}\right).
\]
Therefore, from the above equality and  (\ref{pquadrat}), the
inequality (\ref{JQP}) is equivalent to
$$\frac{n!}{P_n^2(x)}<\frac{(2n-1)!!}{x^{2n}},$$
which is evident.
\end{proof}

\begin{remark}
\label{Jneg} Note that $J_{2n}$ is negative in an interval
$[0,\alpha_n)$ for an increasing sequence of positive
numbers $\{\alpha_n\}_n$; on the contrary,
$Q_n(x)/P_n(x)
>0$ for $x>0$, see Figure \ref{grafic}.
This is important when we want to bound expressions  involving
$f$; see the proof of Theorem~\ref{main2}.
\end{remark}

\begin{figure}[htb]
\centering
 \includegraphics[width=8cm]{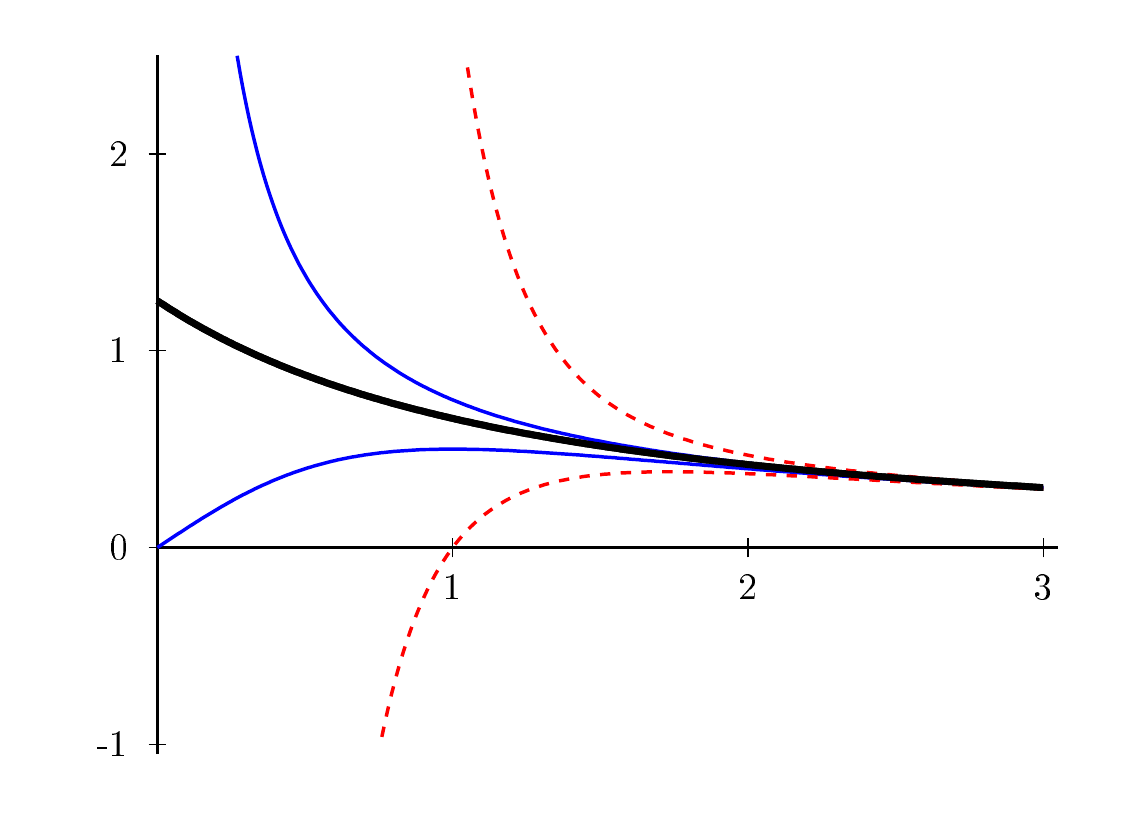}
\caption{Bold solid line: function $f$. Upper and lower light solid
lines: rational functions $Q_3/P_3$ and  $Q_2/P_2$, respectively.
Upper and lower dashed lines: rational functions $J_3$ and $J_2$
respectively.} \label{grafic}
\end{figure}

\section{An approximation procedure}\label{se:3}

\subsection{Taylor formula}

To approach $f$ in a neighborhood of the origin we need to compute
its Taylor expansion  at $x=0.$ The alternating  of signs of the
derivatives $f^{(n)}$ implies that the Taylor series of $f$ at zero
enjoys very nice properties. Let  $T_n$ be the Taylor
polynomial of $f$ of order $n$ at zero.
 Using the
Lagrange remainder, for all $x\ge 0$
$$f(x)=T_n(x)+\frac{1}{(n+1)!}\,f^{(n+1)}(x_0)\, x^{n+1},$$
for some $x_0\in [0,x]$. Then, thanks to the sign of
$f^{(n+1)}(x_0)$ given in Theorem \ref{main}, if $n$ is odd, for
$x>0$, $T_n(x)<f(x)<T_{n+1}(x).$ What is more, given that if $n$ is
odd, $f^{(n)}$ is increasing and
 $f^{(n+1)}$ is decreasing, and if $n$ is even happens the contrary,
the remainder of $T_n$ and $T_{n+1}$ can be expressed in a more
convenient form,
 and
$$f(x)=T_n(x)+\xi_n\, \frac{1}{(n+1)!}\,f^{(n+1)}(0)\, x^{n+1},$$
where $\xi_n\in (0,1).$ Then it is said that the Taylor
 series is  {\it enveloping} (see Small \cite[Section 2.3]{small10}).

The coefficients of the Taylor series are  computed from
$f^{(n)}(0)=P_n(0)f(0)+Q_n(0)$ and are
$$f^{(n)}(0)=\begin{cases}
-(n-1)!!,& \text{if $n$ is odd},\\
(n-1)!!\, \dfrac{\sqrt{2\pi}}{2},& \text{if $n$ is even}.
\end{cases}$$
 Notice that these values of $f^{(n)}(0)$ can
also be obtained from equality~\eqref{lap}. We summarize the above
results in the following lemma:

\begin{lemma}
The Taylor series of $f$ at $0$ is convergent in the whole $[0,\infty)$
and it  is given by
$$f(x)=\frac{\sqrt{2\pi}}{2}-x+\frac{\sqrt{2\pi}}{4}\, x^2-\frac{1}{3}\,x^3+
\frac{\sqrt{2\pi}}{16}\, x^4-\frac{1}{15}\,x^5+\cdots,$$ where the
coefficient of $x^n$ is
\begin{equation*}
\begin{cases}
 \dfrac{\sqrt{2\pi}}{2\, n!!},&  \text{if $n$ is even},\\
\\
 -\dfrac{1}{n!!}, & \text{if $n$ is odd},
\end{cases}
\end{equation*}
and it is enveloping. In particular, for all $n\ge1$ and  for every
$x>0$, $T_{2n-1}(x)<f(x)<T_{2n}(x)$ and the remainder of the Taylor
polynomial is bounded in absolute value by the first neglected term.
\end{lemma}

Combining this  lemma and Theorem~\ref{main} we get the following proposition:

\begin{proposition}\label{propmain} For any positive integer numbers $k,\ell,m,n$
and $x>0$ it holds that
\[
\max\left(T_{2k-1}(x),\frac{Q_{2\ell}(x)}{P_{2\ell}(x)}\right)
<f(x)<\min\left(T_{2m}(x),\frac{Q_{2n-1}(x)}{P_{2n-1}(x)}\right).
\]
\end{proposition}

\subsection{The approximation procedure in action: $1/f$ is strictly convex}
\label{Sub:Main}
As a first application we  present a new proof of the following theorem:

\begin{theorem}\label{main2}
The reciprocal function of  Mills ratio  $1/f$ is strictly convex.
\end{theorem}

As we will see, the proof  will be reduced to show
that the function $g(x):=2+x^2f^2(x)-f^2(x)-3x f(x)$ is strictly
positive
 for all $x>0.$ To  prove this,
we will search for suitable piecewise rational functions,
$\underline R$ and $\overline R$, with coefficients in
$\mathbb{Q}^+\cup\{0\}$ and well defined on $[0,\infty),$ such that
for all $x>0$ it holds that
\[
0<\underline R(x)< f(x)< \overline R(x).
\]
Define $\widehat g(x):=2+x^2\,\underline{R}^2(x)-\overline
{R}\,^2(x)-3x \overline R\,(x)$. Then $g(x)>\widehat g(x)$ and
proving that for $x>0$, $\widehat g(x)>0$, the result will follow.

The fact that all the coefficients are rational is  crucial  in our
approach because the numerator of $\widehat g$ will be a
polynomial with rational coefficients and then rigorous analytic
methods, like  Sturm Theorem,  can be used to prove the inequality.
We recall Sturm Theorem in the Appendix.

These piecewise functions will be constructed from the class of
functions appearing in the statement of Proposition~\ref{propmain}.
More concretely, to get $\underline R$ and $\overline R$  we will
use modified Taylor polynomials in $[0,1]$ and the fractions
$Q_n/P_n$ given in the remainder unbounded interval.

\begin{proof}[Proof of Theorem~\ref{main2}]

 We have that
$$\bigg(\frac{1}{f(x)}\bigg)''=\frac{2+x^2f^2(x)-f^2(x)-3x f(x)}{f^3(x)}=\frac{g(x)}{f^3(x)}.$$
It suffices to prove that $g(x)>0$ (this was Birnbaum conjecture \cite{Bir50}).  As we have already explained, we
will bound below $g$ by a strictly positive function in two steps:
for $x\in[0,1]$ using Taylor formulas, and for $x>1$ using adequate
rational functions $Q_n/P_n$.

\bigskip

\noindent{\bf Step 1. $x\in[0,1]$.}  By Proposition \ref{propmain},
$T_7(x)<f(x)<T_8(x).$ Since the polynomials $T_n$ involve the
irrational number $\sqrt{2\pi}$ in the positive coefficients, we
look for  convenient rational approximations, which are
\begin{equation}
\label{aproRa} \frac{5}{2}<\sqrt{2\pi}<\frac{188}{75}.
\end{equation}
Such fractions are obtained computing the continued fraction of
$\sqrt{2\pi}$, $[2,1,1,37,4,1,1\ldots]$ and the corresponding
convergents $2,3,5/2,{188}/{75},{757}/{302},\ldots$

\begin{figure}[htb]
\centering
\includegraphics[width=8cm]{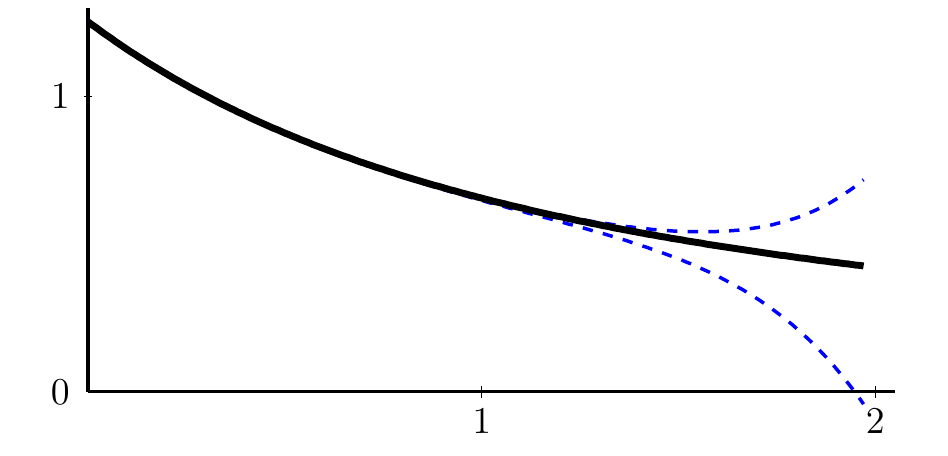}
\caption{Solid line: function $f$.  Upper and lower dashed lines:
polynomials  $T_{8,u}$ and $T_{7,\ell}$ respectively.} \label{fig2}
\end{figure}

Denote by $T_{7,\ell}(x)$ (respectively, $T_{8,u}(x)$) the
polynomial $T_7(x)$ (resp. $T_8(x)$) with the fraction in the left
hand side of (\ref{aproRa}) (respectively in the right hand side) instead
of the number $\sqrt{2\pi}$. We have that for $x\in[0,1]$,
$$0<T_{7,\ell}(x)<T_7(x)<f(x)<T_8(x)<T_{8,u}(x),$$
where \begin{align*} T_{7,\ell}(x)&=-{\frac {1}{105}}{x}^{7}+{\frac
{5}{192}}{x}^{6}-\frac1{15}{x}^{5}+{ \frac
{5}{32}}{x}^{4}-\frac13{x}^{3}+\frac58{x}^{2}-x+\frac54,
\\[1.mm]
T_{8,u}(x)&={\frac {47}{14400}}\,{x}^{8}-{\frac
{1}{105}}{x}^{7}+{\frac {47}{ 1800}}{x}^{6}-\frac1{15}{x}^{5}+{\frac
{47}{300}}{x}^{4}-\frac13{x}^{3}+ {\frac {47}{75}}{x}^{2}-x+{\frac
{94}{75}},
\end{align*}
see Figure~\ref{fig2}. We first check that $T_{7,\ell}(x)>0$ for
$x\in[0,1]$. This is done using that $T_{7,\ell}$ is a polynomial
with rational coefficients, and thus its roots can be studied by
finite algorithms. Specifically, by using Sturm Theorem (see the
appendix), the number of real roots of
 such a polynomial
in an interval with rational extremes, or in an infinite interval,
can be computed. In the case of $T_{7,\ell}$ there are no real roots
in the interval $[0,1]$; actually, the first positive real root can
be located at the interval $(31/16,2)$, as Figure \ref{fig2}
illustrates. Since $T_{7,\ell}(0)>0$, it follows
 the strict positivity of $T_{7,\ell}$ on $[0,1]$.

 Now, in the expression of $g(x)$, change
$f(x)$ in the terms with  positive sign  by $T_{7,\ell}(x)$, and by
$T_{8,u}(x)$ in the terms with negative sign. We get the polynomial
\[
G(x)=2+x^2T_{7,\ell}^2(x)-T_{8,u}^2(x)-3x T_{8,u}(x),
\]
which is
\begin{align*}
G(x)=&{\frac {813359}{10160640000}}\, {x}^{16}-{\frac {41}{94500}}
\,{x}^{15}+ {\frac {17139569}{10160640000}}\, {x}^{14}-{\frac
{139}{25200}}\, {x}^{ 13}+{\frac {1158121}{72576000}}
\,{x}^{12}\\&-{\frac {15671}{378000}} \,{x} ^{11}+{\frac
{1308953}{13440000}} \,{x}^{10}-{\frac {327233}{1512000}}
\,{x}^{9}+{\frac {1528967}{3780000}}\, {x}^{8}-{\frac
{9941}{14000}}\, { x}^{7}+{\frac {616499}{540000}}\,
{x}^{6}\\&-{\frac {1862}{1125}} \,{x}^{5} +{\frac {189937}{90000}}\,
{x}^{4}-{\frac {1031}{450}} \,{x}^{3}+{\frac {179249}{90000}}\,
{x}^{2}-{\frac {94}{75}}\, x+{\frac {2414}{5625}},
\end{align*}
and it satisfies $$g(x)>G(x)>0, \ \text{for $x\in[0,1]$}.$$ As
before, the strict positivity of $G$ on $[0,1]$ is  proved observing
that $G(0)>0$, and that the polynomial $G$ has no real roots
 in that interval;  in fact $G$ has exactly two real  roots and the smallest
 one is in $(11/10,12/10).$

Notice that the trick of replacing $\sqrt{2\pi}$ by upper or lower
rational bounds is  crucial  because it allows the use of the
aforementioned approach.

\bigskip

\begin{figure}[h]
\centering
\includegraphics[width=8cm]{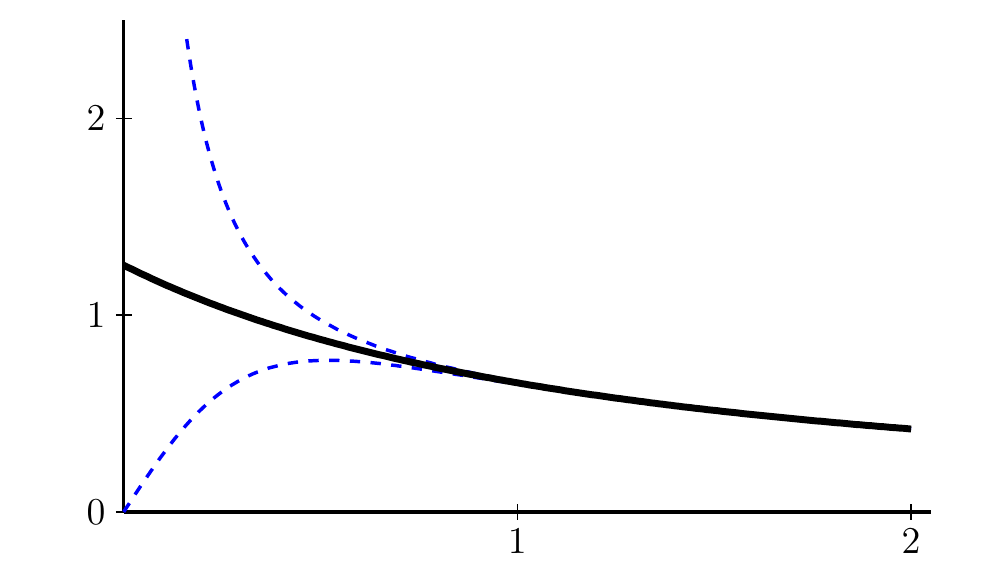}
\caption{Solid line: function $f$.  Upper and lower dashed lines:
rational functions $Q_{11}/P_{11}$ and $Q_{10}/P_{10}$ respectively.} \label{fig3}
\end{figure}

\noindent{\bf Step 2. $x>1$.}
 By Proposition \ref{propmain} we approximate $f$ by
${Q_{10}(x)}/{P_{10}(x)}<f(x)<{Q_{11}(x)}/{P_{11}(x)},$
 where
\begin{align*}
\frac{Q_{10}(x)}{P_{10}(x)}&={\frac {{x}^{9}+44 {x}^{7}+588
{x}^{5}+2640 {x}^{3}+2895 x}{{x}^{ 10}+45 {x}^{8}+630 {x}^{6}+3150
{x}^{4}+4725 {x}^{2}+945}},
\\[1.mm]
\frac{Q_{11}(x)}{P_{11}(x)}&={\frac {{x}^{10}+54 {x}^{8}+938
{x}^{6}+6090 {x}^{4}+12645 {x}^{2} +3840}{{x}^{11}+55 {x}^{9}+990
{x}^{7}+6930 {x}^{5}+17325 {x}^{3}+ 10395 x}} ,
\end{align*}
see Figure~\ref{fig3}. Then in $g$  change $f$ by $Q_{10}/P_{10}$ in
the terms with positive signs, and by $Q_{11}/P_{11}$ in the terms
with negative sign. We obtain a rational function with positive
denominator,  and numerator
\begin{align*}
N(x)= &\, {x}^{36}+185 {x}^{34}+15388 {x}^{32}+761580
{x}^{30}+25019940 {x}^ {28}+576522420 {x}^{26}\\&+9601604100
{x}^{24}+117398708820 {x}^{22}+ 1059855272550 {x}^{20}+7043405005350
{x}^{18}\\&+33995881448100 {x}^{ 16}+115607852356500
{x}^{14}+259703297525700 {x}^{12}\\&+ 329529066520500
{x}^{10}+108511796893500 {x}^{8}-233411033740500 {x
}^{6}\\&-247669566519375 {x}^{4}-66176702274375
{x}^{2}-6584094720000.
\end{align*}

So it suffices to prove that $N(x)>0$ for $x>1$. This again is  a
consequence of the  Sturm  Theorem, which implies that  $N$ has no
real roots in $(1,\infty)$ since  its biggest root is in the
interval $(93/100 ,94/100)$, and $N(1)>0$. \end{proof}

\subsection{A second application: Bounds involving square roots}
\label{Sub:square}
As a second application we study lower and upper bounds of $f$ of the form
\begin{equation}
\label{h}
\psi_{a,b,c}(x)=\frac{a}{\sqrt{x^2+b}+cx}, \ \text{with $a,\, b,\, c>0$}.
\end{equation}
The asymptotic expansion at infinity of such a function is
\begin{equation}
\label{asymptoticPsi}
\psi_{a,b,c}(x)\sim \frac{\gamma_1}{x}+\frac{\gamma_3}{x^3}+\frac{\gamma_5}{x^5}+\cdots, \ x\to\infty,
\end{equation}
for $\gamma_i=\gamma_i(a,b,c), \ i=1,3,\dots$, which is of  the same type of the expansion of $f$ (see (\ref{asy_feller})).
Well known bounds
of this class are, for $x\ge 0$,
\begin{equation}
\label{FitaSampford}
\frac{2}{\sqrt{x^2+4}+x}<f(x)<\frac{4}{\sqrt{x^2+8}+3x}.
\end{equation}
The lower bound was proved by Birnbaum \cite{Bir42};
it is also   given in Itô and McKean \cite[Page 17]{ItoMc74},
jointly with a worst  upper bound, and
  attributed to Y. Komatu
(1955) (see the reference therein).
The  upper bound is equivalent to the strict convexity of $f$ (see the proof
of Theorem \ref{main}), and, as we commented,
it was proved   by Sampford \cite{Sam}
and Shenton \cite{se}.

However, these bounds have at 0 different values than $f(0)$,
and then  Boyd \cite{Boyd59} (reproduced
in Mitrinovi\'{c} \cite[p. 179]{Mitri70}; see also Amos
\cite[Inequalities (12)]{Amos73}) gives the following new bounds, for $x>0$,
$$\frac{\pi}{\sqrt{x^2+2\pi}+(\pi-1)x}<f(x)<\frac{\pi}{\sqrt{(\pi-2)^2x^2+2\pi}+2x}.$$
These bounds, denoted by $\ul \psi$ and $\ol \psi$ respectively,  satisfy that
$$\ul\psi(0)=\ol\psi(0)=f(0) \quad \text{and} \quad \lim_{x\to\infty}x\big(\ul \psi(x)
-f(x)\big)=\lim_{x\to\infty}x\big(\ol \psi(x)-f(x)\big)
=0,$$
and they are the sharpest bounds of the form $\psi_{a,b,c}$ that satisfy the above conditions.

\bigskip

To systematize and compare these bounds we introduce some notation: We will consider a generic
function $g$ that is enough regular at 0 an that have an asymptotic expansion at infinity as the one given in
(\ref{asymptoticPsi}).  We say the such a
function $g$
 is  equal to $f$ at 0 of order $i\ge 1$ if
$g$ and $f$, and its derivatives up to order $i-1$ coincide at 0, that is,
$$g^{(k)}(0)=f^{(k)}(0), \text{for $k=0,\dots,i-1$}.$$
We say that $g$ and $f$  are equal at 0 of order 0 if there is no
condition of the values of $g$ and $h$ (and its derivatives) at 0.
In a similar way, we say that $g$ and $f$ are  equal at infinity of order $j\ge 1$
 if
 $$\lim_{x\to \infty}x^k\,(g(x)-f(x)\big)=0, \text{for $k=0,\dots,2j-1$},$$
 and that they are equal   of order 0 if there is no restriction on the behaviour
 at infinity. Finally, for $i,j\ge 0$, we say that  $g$ and $f$ are equal
   of order $(i,j)$ if they are equal at 0 of order $i$ and at infinity at order $j$.

 Particularizing these notations to our bounds, since $\psi_{a,b,c}$
has three free parameters, we introduce four functions to study the different
possibilities of equality
between $\psi_{a,b,c}$ and $f$. For $i,j\in \{0,1,2,3\}$ such
that $i+j=3$, we denote by $W_{i,j}$ the function $\psi_{a,b,c}$ which is equal to
$f$ of order $(i,j)$.
The four functions are
$$W_{3,0}(x)=\frac{\pi}{\sqrt{2x^2(4-\pi)+2\pi}+2x},
\quad  \, W_{1,2}(x)=\frac{\pi}{\sqrt{x^2+2\pi}+(\pi-1)x},
$$
 $$W_{2,1}(x)=\frac{\pi}{\sqrt{(\pi-2)^2x^2+2\pi}+2x}\quad \text{and}\quad W_{0,3}(x)=\frac{4}{\sqrt{x^2+8}+3x}.$$
 Note that $W_{1,2}=\ul \psi$ and $W_{2,1}=\ol \psi$,
and $W_{0,3}$ coincides with the upper bound in (\ref{FitaSampford}).
$W_{3,0}$ seems to be new. For the deduction of these functions see the proof of point 2
in the next
theorem:

\begin{theorem}

\label{ThW}
\rule[0mm]{0mm}{0mm}

\begin{enumerate}[\bf 1.]
\item For $x>0$,
$$\max\big\{W_{3,0}(x),\, W_{1,2}(x)\big\}<f(x)<\min\big\{W_{0,3}(x),\, W_{2,1}(x)\big\}.$$

\item For $(i,j)=(0,2)$ or $(i,j)=(2,0)$,  the functions
$W_{i+1,j}$ and $W_{i,j+1}$  are the sharpest lower and upper bound of $f$
of the form $\psi_{a,b,c}$ such that are equal to $f$ of order $(i,j)$.
Moreover,  the functions  $W_{1,2}$ and $W_{2,1}$
  are the sharpest lower and upper bound of $f$
of the form $\psi_{a,b,c}$ such that are equal to $f$ of order $(1,1)$

\item Between the possible combination of upper and lower bounds,
 $W_{2,1}$ and $W_{3,0}$ are optimal in the sense that
  $\max_{x\ge 0}\big(W_{2,1}(x)-W_{3,0}(x)\big)$
 is minimal. In particular,
 $$\max_{x\ge 0}\Big(W_{2,1}(x)-W_{3,0}(x)\Big)<0.015.$$

\end{enumerate}

\end{theorem}

\begin{proof}

\rule[0mm]{0mm}{0mm}

\noindent{\bf 1.} We always assume $x>0$. This part refers to four inequalities. As we
already commented,
 $f(x)<W_{0,3}(x)$
is equivalent to
$$2+x^2f^2(x)-f^2(x)-3x f(x)>0$$
and we proved it in Theorem \ref{main}. We now prove $W_{3,0}(x)<f(x)$, which is new,
and we omit the proofs of the other two inequalities, that are very similar
(and, indeed, these bounds are known).

The inequality $W_{3,0}(x)<f(x)$ is equivalent to
\begin{equation}
\label{primerasubs}
4x^2f^2(x)-2\pi x^2 f^2(x)+2\pi f^2(x)+4\pi xf(x)-\pi^2>0.
\end{equation}
The main difficulty in this proof is the apparition of $\pi$ and $\pi^2$ in the above
inequality, and it is not convenient to change these number by rational approximations
till the last moment. The proof has two steps:

\smallskip

\noindent{\it First step: $x\in (0,1]$}. We will use the Taylor polynomials
$T_7$ and $T_8$ as in the proof of Theorem \ref{main}. Specifically,
 in the left hand side of (\ref{primerasubs}) change $f$ by $T_7$ in positive
 terms and by $T_8$ in the negative terms. We get a polynomial of degree 18 where
 some coefficients are multiplied by $\sqrt{\pi/2}, \, \sqrt{2/\pi},\, \pi$ or $1/\pi$. Now in the
 terms with $\sqrt{\pi/2}, \, \sqrt{2/\pi}$  change them by fractions
 (in agreement with the sign)
 by  using the convergents of $\sqrt{\pi/2}$ given by
 \begin{equation}
 \label{convarrelpi}
 \frac{851}{679}<\sqrt{\frac{\pi}{2}}<\frac{94}{75},
 \end{equation}
 and finally,
 change $\pi$ and $1/\pi$ by  using
\begin{equation}
\label{convPi}
\frac{333}{106}<\pi<\frac{355}{113}.
\end{equation}
 We get  a polynomial with rational coefficients that by Sturm Theorem is strictly positive in $(0,1]$.

\smallskip

\noindent{\it Second step: $x>1$}.
In the expression of the left hand side of (\ref{primerasubs}), first substitute $f$ in the positive terms
by $Q_{12}/P_{12}$ and in the negative terms by $Q_{13}/P_{13}$. This gives a rational function
with positive denominator. The numerator is a  polynomial
of order 48 with integer coefficients  some of them multiplied by $\pi$ or $\pi^2$.
In that expression, use the convergents of $\pi$ given in (\ref{convPi}).
 We arrive to a polynomial with rational
coefficients that, by Sturm Theorem, is strictly positive for $x>1.$

\bigskip

\noindent{\bf 2.} We prove that $W_{3,0}$ and $W_{2,1}$ are the sharpest upper
and lower bounds of $f$
of the form $\psi_{a,b,c}$ such that are equal to $f$ of order $(2,0)$.
Later, we comment the proofs of the other cases.

First we consider the family of functions $\psi_{a,b,c}$ such that are equal to $f$ of order $(2,0)$.
 From $f(0)=\sqrt{\pi/2}$ and $f'(0)=-1$ we deduce that $a=\pi c/2$ and $b=\pi c^2/2$.
Hence we can parametrize that family by
$$\psi_{c}(x)=\frac{\pi c/2}{\sqrt{x^2+\pi c^2/2}+cx},\ c>0.$$
Fixed $x>0$, by derivation with respect to $c$, it is proven that the function
$c\in(0,\infty)\to \psi_c(x)$ is strictly increasing. Hence, for the optimal lower bound we should take
$c$ as large as possible satisfying $\psi_c(x)<f(x)$, and for the optimal upper bound,
$c$ as small as possible with $f(x)<\psi_c(x)$.

Imposing that $\psi_c$ is equal to $f$ of order $(3,0)$, that means,
$\psi_c''(0)=f''(0)=\sqrt{\pi/2},$ we deduce that $c$ should be
$c_0=\sqrt{2/(4-\pi)}$. Note that $\psi_{c_0}=W_{3,0}$.
On the other hand, to look $c$ such that  $\psi_c$ is equal to $f$ of order
(2,1),   consider the asymptotic expansion of $\psi_c$ when
$x\to \infty$,
$$\psi_c(x)\sim \frac{\pi c}{2+2c}\,\frac{1}{x}+\cdots.$$
Comparing with the asymptotic expansion of $f$ (see (\ref{asy_feller}))
we should take $c_1=2/(\pi-2)$. Then $\psi_{c_1}=W_{2,1}$.
Note that $c_0<c_1$.

By  Taylor formula for
$\psi_c$ it is deduced that
$$\lim_{x\to 0}\frac{f(x)-\psi_c(x)}{x^2}=
\frac{(\pi-4)c^2+2}{2\sqrt{2\pi}c^2}.$$
Hence, if $c>c_0$, that limit is negative, and then, for $x$ near 0, $f(x)<\psi_c(x)$.
So $\psi_c$ is not a lower bond.

Now, from the asymptotic expansion when $x\to\infty$ of $\psi_c$ and $f$ it follows that
$$\lim_{x\to\infty}x\big(f(x)-\psi_c(x)\big)=\frac{(2-\pi) c+2}{2+2c}.$$
Then, if $c<c_1,$ we have that for $x$ large enough,  $f(x)>\psi_c(x)$, so $\psi_c$
is not an upper bound.

To summarize, by point {\bf 1}, $\psi_{c_0}(x)<f(x)<\psi_{c_1}(x)$, and  for $c\in (c_0,c_1)$, $\psi_c$ is neither a lower bound
nor an upper bound, Then
the proof is complete.

For the case $W_{1,2}$ and $W_{1,3}$, the functions $\psi_{a,b,c}$ such that
has a coincidence of order (1,1) with $f$ can be parametrized as
\begin{equation}
\label{eta}
\eta_a(x)= \frac{a}{\sqrt{x^2+2a^2/\pi}+(a-1)x}, \ a>1.
\end{equation}
The function $W_{1,2}$ is the case $a_0=\pi$, and $W_{2,1}$ is
$a_1=\pi/(\pi-2)$. Fixed $x>0$, The function $a\to \eta_a(x)$ is strictly decreasing, and
for $a\in\big(\pi/(2-\pi),\pi\big)$, the function $\eta_a$ is  neither a lower bound
nor an upper bound.

Finally, consider the case $W_{1,2}$ and $W_{0,3}$. The functions $\psi_{a,b,c}$ such that
has a coincidence of order (0,2) with $f$ can be parametrized as
\begin{equation}
\label{chi}
\chi_c(x)=\frac{1+c}{\sqrt{x^2+2(1+c)}+cx}.
\end{equation}
The function $\chi_c$ with a coincidence of order (1,2) with $f$ corresponds to
$c=\pi-1$, and it is $W_{1,2}$. The function $\chi_c$ with a coincidence of
order (0,3) with $f$ corresponds to $c=3$ and it is $W_{0,3}$. Using a similar
argument it is proved that for $c\in (\pi-1,3)$, $\psi_c$ is  neither a lower bound
nor an upper bound.

\bigskip

\noindent{\bf 3.} We first will prove that  $\sup_{x\ge 0}\big(W_{2,1}(x)-W_{3,0}(x)\big)<0.015.$
To this end, write $W(x)=W_{2,1}(x)-W_{3,0}(x)$,
$$W_{2,1}(x)=\frac{\pi}{\sqrt{(\pi-2)^2x^2+2\pi}+2x}=\frac{\pi}{\Delta_0(x)+2x},$$
and
$$W_{3,0}(x)=\frac{\pi}{\sqrt{2x^2(4-\pi)+2\pi}+2x}=\frac{\pi}{\Delta_1(x)+2x},$$
A computation shows that
$$W'(x)=\frac{\pi x M(x)}{\Delta_0(x)\Delta_1(x)\big(\Delta_0(x)+2x\big)^2\big(\Delta_1(x)+2x\big)^2},$$
where $M(x)$ is a polynomial. So the zeroes of $M(x)$ will give the behaviour of $W$.
Squaring conveniently the equation $M(x)=0$, which includes radicals,
 it can be translated to a polynomial equation and we have that if $x_0$ is a zero of
 $M$ then it is also a zero of a certain polynomial
say, $R_{12}$ (the reciprocal is not true); that polynomial has the form
$$R_{12}(x)=a_6 x^{12}+a_5x^{10}+a_4x^8+a_3x^6+a_2x^4+a_1x^2+a_0,$$
where
$a_6,a_5,a_4,a_3,a_0>0$ and $a_2,a_1<0$. By Descartes Theorem, $R_{12}$ will have two positive zeroes
or none. Moreover, $R_{12}(0)>0,\ R_{12}(17/10)<0,\, R_{12}(\infty)>0$, and by Bolzano
theorem $R_{12}$ has exactly two real zeroes, say $x_1$ and $x_2$. Moreover,
$$R_{12}(15/10)R_{12}(16/10)<0 \quad \text{and}\quad R_{12}(198/100)R_{12}(199/100)<0,$$
and
$$M(1)>0,  \ M(17/10)>0 \quad \text{and}\quad M(2)<0.$$
This implies that $x_1$ is not a zero of $M$ ($x_1$ is a spurious zero introduced
by the procedure to cancel the radicals). To summarize, $W$ has a maximum at $x_2\in
I:=(198/100,199/100).$ Finally, both $W_{2,1}$ and $W_{3,0}$ are strictly decreasing,
so
$$\max_{x\ge 0}\Big(W_{2,1}(x)-W_{3,0}(x)\Big)=
\max_{x\in I}\Big( W_{2,1}(x)-W_{3,0}(x)\Big)\le W_{2,1}(198/100)-W_{3,0}(199/100)<0.015.$$

\medskip

Now it is easy to check that there are $y_1,y_2,y_3\ge 0$ such that
$W_{2,1}(y_1)-W_{1,2}(y_1)>0.015$, $W_{0,3}(y_2)-W_{1,2}(y_2)>0.015$
and $W_{0,3}(y_3)-W_{3.0}(y_3)>0.015$.

\end{proof}

\begin{remarks}

\rule[0mm]{0mm}{0mm}

\begin{enumerate}[\bf 1.]
\item Note that from the expression (\ref{eta}) and the  properties of the function
$\eta_a$,
taking $a=2$ we get  a nice and simple  upper bound of $f$.
Specifically,
$$\eta_2(x)=\frac{2}{\sqrt{x^2+8/\pi}+x}.$$
This bound was found by Pollak \cite{Poll56} (see also
 Mitrinovi\'{c} \cite[p. 179]{Mitri70}).

\item In a similar way, from the expression (\ref{chi})
a lower bound for $f$, less sharp than $W_{1,2}$,
but more friendly, is given by taking $c=2$, which gives
$$\chi_2(x)=\frac{3}{\sqrt{x^2+6}+2x}.$$

\item The three parameters of $\psi_{a,b,c}$ allow to build alternative bounds adapted
to more specific purposes. For example, we construct   lower and upper bounds
  that approximate sharply  $f$ in a neighborhood of $x=2$; they are based in a function with
  a coincidence with $f$ of order (1,1) and later manipulating the coefficients
  in order that
$a,b,c$ to be rational numbers: for $x>0$,
$$\frac{200}{\sqrt{1521 x^2+25600}+161x}<f(x)<\frac{192}{\sqrt{4225x^2+20736}+127x}.$$
The proof is analogous to the proofs of the   bounds in Theorem \ref{ThW}
\end{enumerate}
\end{remarks}

\subsection{More bounds with rational  or exponential functions}
The topic of searching new bounds for the Mills ratio seems inexhaustible, and
we would like to mention the papers of Dümbgen \cite{Dum10} and
Avram \cite{Avram13}  where new methodologies
are presented. To finish we
study five families of  bounds  and  we just  emphatize that  all inequalities
are proved by the same method.

\bigskip

\noindent{\bf 1. Padé approximations at the origin.} In Corollary \ref{coromain}
we proved that the rational fractions $Q_n/P_n$ are the Padé--Laurent approximations
at infinity. Now we will consider  Padé approximations at the origin.
Standard computations show that (with the usual notations for Padé aproximations)
\begin{equation*}
p_{0,1}(x)=\frac{\pi}{\sqrt{2\pi}+2x}
\quad\text{and}\quad p_{1,2}(x)= \frac{6\pi\sqrt{2\pi}-24\sqrt{2\pi}
+(48-16\pi)x}{12\pi-48-4\sqrt{2\pi}x+2(8-3\pi)x^2}.
\end{equation*}
Using our procedure, it is proven that for $x>0$,
$$p_{1,2}(x)<f(x)<p_{0,1}(x).$$
It is worth to comment that $p_{0,1}(x)$ is very simple and it is a
 good global approximation to $f$.
Actually, $p_{0,1}(x)-p_{1,2}(x)<0.08$
\bigskip

\noindent{\bf 2. Bounds with simple rational functions.}
We construct bounds of the type $a/(b+x)$. Note  that
such  functions have an expansion at infinity with nonzero terms in the even
coefficients, so we need to change slightly our notations
 of Section \ref{Sub:square}; here and in  next points,
 we say that $g$ is equal to $f$ at infinity of order $j\ge 1$
 if $\lim_{x\to\infty}x^k\big(g(x)-f(x)\big)=0$ for $k=0,\dots,j$
 (the notation for the equality at 0 does need to be modified).
We consider only the most interesting cases. Let
$U_{2,0}$ equal to $f$ of order (2,0) and $U_{1,1}$ equal
of order (1,1). They are
$$U_{2,0}(x)=\frac{\pi}{\sqrt{2\pi}+2x}\quad \text{and}\quad
U_{1,1}(x)= \frac{\pi}{\sqrt{2\pi}+\pi x}.$$
Note that $U_{2,0}=p_{0,1}$ of previous point.
We consider also approximations to both function with rational coefficients.
It is proved that for $x>0$,
$$\frac{105}{91+110x}<U_{1,1}(x)<f(x)<U_{2,0}(x)<\frac{44}{35+28x}.$$
We have that $U_{2,0}(x)-U_{1,1}(x)<0.15$, and the difference
between the corresponding bounds with rational coefficients is lower than 0.19.

\bigskip

\noindent{\bf 3. Bounds with quadratic rational functions.} Following an idea
of Bryc \cite{Bryc02} we consider
 functions of the form
$V(x)=(a+bx)/(c+dx+x^2).$
 We study   some of the functions $V_{i,j}$ with $i+j=4$  such
that are equal  to $f$ of order $(i,j)$. First we consider  the functions  $V_{2,2}$ and
$V_{1,3}$, and some corresponding  functions with simple
rational coefficients. We have
\begin{align*}
 \frac {35+15x}{28+37x+16x^2}<V_{2,2}(x)=&
\frac{ \sqrt{2\pi }+(\pi -2) x}{2+x \sqrt{2\pi}+  ( \pi -2 ) x^2}<f(x)\\
&\qquad \qquad<V_{1,3}(x)=\frac{\sqrt{2\pi}+2x}{2+x \sqrt{2\pi} +2x ^2}
<\frac{2}{5}\,\frac{13+10x}{4+5x+4{x}^{2}}.
\end{align*}
The bounds are quite good: $V_{1,3}(x)-V_{2,2}(x)<0.07$, and the difference
between the corresponding bounds with rational coefficients is lower than 0.13.

The function $V_{3,1}$ is also interesting. Its expression is
$$V_{3,1}(x)=\frac{\sqrt{2\pi}(\pi-2)+(4-\pi)x}{2(\pi-2)+\sqrt{2\pi}x+(4-\pi)x^2}.$$
It is given in Bryc \cite{Bryc02} (see also  Avram \cite{Avram13}) as a good uniform
 approximation to $f$
 without a formal
proof. Following our procedure it is proven that it is an upper bound of $f$,
and indeed better than $V_{1,3}$ since $V_{3,1}(x)-V_{2,2}(x)<0.015$.

\bigskip

\noindent{\bf 4. Bounds involving one exponential term.}
Inspired by Karagiannidis and Lioumppas \cite{KaraLio07}
 we consider upper and lower bounds of $f$ of the form
$$\kappa_{a,b}(x)=\frac{1-e^{-ax}}{bx},$$
for $a,b>0$.
In Karagiannidis and Lioumpas \cite{KaraLio07} the values $\bar a=1.98/\sqrt 2$ and $\bar b=1.135$ are proposed
to get  a good approximation to $f$ based on numerical arguments. Following the notations
introduced at point {\bf 2} we consider the functions   $Z_{2,0}$ and $Z_{1,1}$
given by
$$Z_{2,0}(x)=\frac{1-e^{-4x/\sqrt{2\pi}}}{4x/\pi} \quad \text{and}\quad Z_{1,1}(x)=\frac{1-e^{-\sqrt{2\pi}\,x/2}}{x}.$$
Combining the approximations of $f$ with the properties of a function of the form
$1-{\rm exp}(-ax)$,
the study of inequalities between $f$ and $\kappa_{a,b}$ are reduced to prove the strict
positivity of certain polynomials with rational coefficients, as in the previous sections.
It is proved that for $x>0$,
$$Z_{2,0}(x)<f(x)<Z_{1,1}(x).$$
These bounds are quite good since $Z_{1,1}(x)-Z_{2,0}(x)<0.1$.
Moreover, following the same arguments as in  Theorem \ref{ThW} it is proved
that $Z_{2,0}$ and $Z_{1,1}$ are the sharpest lower and upper bound of the
form $\kappa_{a,b}$ such that $\lim_{x\to 0}\kappa_{a,b}(x)=f(0).$
The function $\kappa_{\bar a,\bar b}$ proposed by Karagiannidis and Lioumpas \cite{KaraLio07}
is neither a lower bound nor an upper bound.

\bigskip

\noindent{\bf 5.  Chernoff type bounds for the Gaussian $Q$-function.}
In the engineering literature
the function
$$Q(x)=\frac{1}{\sqrt{2\pi}}\int_x^\infty e^{-t^2/2}\, dt$$
is called the Gaussian $Q$--function. By technical reasons (see Chang {\it et al.}
\cite{Chang11}
and Côté {\it et al.} \cite{Cote12} and  the references therein),
Chernoff type bounds for $Q$ of the form
$$C_{a,b}(x)=a\, e^{-bx^2}, \ x\ge 0,$$
for $a,\, b\ge 0$, are particulary convenient in the analysis of communications systems. Since
$$Q(x)=\frac{1}{\sqrt{2\pi}}\, e^{-x^2/2}f(x),$$
Chernoff type bounds for Mills ratio  give automatically Chernoff type bounds
for $Q$. In particular,  for $x> 0$,
$$e^{-3x^2/5}<f(x)<\sqrt{\frac{\pi}{2}}.$$
Hence,
$$\frac{1}{\sqrt{2\pi}} \, e^{-11x^2/10}<Q(x)<\frac{1}{2}\, e^{-x^2/2}.$$
The  upper bound follows from the fact that $f$ is strictly decreasing
and $f(0)=\sqrt{\pi/2}$; in agreement with Chang {\it et al.}
\cite[Corollary 1]{Chang11}  it is the  optimal Chernoff type upper bound
for $Q$. The lower bound is proved applying our procedure.

\section*{Appendix: Sturm  Theorem}

We present a general version of Sturm Theorem \cite{s} (see
Isaacson and Keller \cite[Page 126]{Isaac94}).
 Let $f:[a,b] \to \R$ be a differentiable function. A sequence
 of continuous functions
on $[a, b]$, $\{f_0,f_1, . . . , f_m\}$, with $f_0=f$,    is called
a Sturm  sequence for $f$ on $[a, b]$ if the following holds:

\begin{enumerate}

\item  $f$ has at most simple roots in $[a,b]$.

\item  $f_m$ does not vanish in $[a, b].$

\item If $f (x_0) = 0$ for some $x_0\in[a, b]$ then $f_1(x_0)f_0'(x_0) >
0.$

\item If for some $i>0,$ $f_i(x_0) = 0$ with $x_0\in[a, b]$ then $f_{i+1}(x_0)f_{i-1}(x_0) <
0.$

\end{enumerate}

\noindent{\bf Sturm  Theorem.} {\it Let $\{f_0,f_1, . . . , f_m\}$
be a Sturm  sequence for $f = f_0$ on $[a, b]$ with $f (a)f
(b)\ne0.$ Then the number of solutions of $f=0$ on $(a, b)$ is equal
to $V (a)-V (b),$ where $V (c)$ is the number of changes of sign in
the sequence $[f_0(c), f_1(c),\ldots, f_m(c)]$.}

\medskip

We remark that if in  $[f_0(c), f_1(c),\ldots, f_m(c)]$ some value
$f_j(c)$ vanishes, then it does not contribute to the number of
changes of sign of the sequence.

When $f$ is a polynomial, Sturm approach enjoys the very useful
properties that it can be applied even without knowing a priori if
the polynomial has simple roots, and that  there is a simple
procedure to construct its Sturm  sequence.  Indeed, if $p$ is a
polynomial of degree $n$ then  define $\{p_0,p_1, . . . , p_m\}$
with $m\le n$ setting $p_0 = p,$ $p_1 = p'$  and
\begin{align*}
p_{i-1}(x) &= q_i(x)p_i(x)-p_{i+1}(x),\quad \mbox{for}\quad i =
1,2,\ldots, 2, \ldots,m-1,\\
 p_{m-1}(x) &= q_m(x)p_m(x),
 \end{align*} where $q_i$ and
$p_{i+1}$ are respectively the quotient and the remainder (the
latter with the sign changed) of the division of $p_{i-1}$ by $p_i$.
The construction of this sequence ends when the remainder is zero,
i.e., $p_{m+1} = 0.$ In this case, since this procedure is
essentially the Euclides algorithm, $p_m$ is the greatest common
divisor of $p_0$ and $p_1$. When all the zeros of $p$ are simple
then $m=n$ and  $p_m$ is a nonzero constant. Then it is easy to show
that $\{p_0,p_1, . . . , p_m\}$ is a Sturm  sequence for $p$ on any
interval. If $p$ has some multiple zeroes then $m<n$. Since $p_m$
divides both $p_0$ and $p_1,$ it also divides $p_i$ for all~$i$.
Then setting $\widetilde p_i = p_i/p_m,$ it happens that $\widetilde
p_0$ has the same zeroes of $p$ but all with multiplicity 1, and
 $\{\widetilde p_0, \widetilde p_1, . . . ,\widetilde
p_m\}$ is a Sturm   sequence for $\widetilde p_0$ on any interval,
and so the localizacion problem of the real zeroes of $p$  is
solved.

In the particular case that the polynomial $p$ has rational
coefficients and $a$ and $b$ are also in $\mathbb{Q}$ then all the
conditions of Sturm Theorem can be checked analytically.  Clearly,
this result can be extended to $a=-\infty$ or $b=\infty$.

As an illustration of the method we give all the details for proving
that the function $g$ appearing in the proof of
Theorem~\ref{main2} is strictly positive in $[0,{45}/{100}].$ We
have that for $x\in [0,{45}/{100}]$,
$$0<-x+\frac 54=T_{1,\ell}(x)<T_1(x)<f(x)<T_2(x)<T_{2,u}(x)=\frac{47}{75}x^2-x+\frac{94}{75}.$$
Define the polynomial
\[
p(x)=2+x^2T_{1,\ell}^2(x)-T_{2,u}^2(x)-3x T_{2,u}(x),
\]
which satisfies that  $p(x)<g(x)$ for $x \in [0,{45}/{100}].$ Then
\[
p(x)=p_0(x)={\frac {3416}{5625}} {x}^{4}-{\frac {469}{150}}
{x}^{3}+{\frac { 179249}{90000}} {x}^{2}-{\frac {94}{75}} x+{\frac
{2414}{5625}}.
\]
Its Sturm sequence is
\begin{align*}
p_1(x)&= {\frac {13664}{5625}} {x}^{3}-{\frac {469}{50}}
{x}^{2}+{\frac { 179249}{45000}} x-{\frac {94}{75}}
  ,\\[1mm]
p_2(x)&={\frac {355316101}{175680000}} {x}^{2}-{ \frac
{3202259}{9369600}} x -{\frac {1135387}{43920000}}
   ,\\[1mm]
p_3(x)&=-{\frac { 45065042306901196}{18035647375691743}}x+ {\frac
{24672388276565440}{18035647375691743}}
  ,\\[1mm]
p_4(x)&=  -{\frac {24548932950879333622396114393201747}{
62423915106233442706008445888230000}}
  .
\end{align*}

Define $S(c)=[\operatorname{sgn}(p_0(c)),
\operatorname{sgn}(p_1(c)),\ldots, \operatorname{sgn}(p_4(c))]$,
with $c\in\mathbb{R}\cup\{-\infty,\infty\}$, where by notation
$\operatorname{sgn}(\pm\infty)=\pm.$

It holds that $S(0)=[+,-,-,+,-]$, $S(45/100)=[+,-,+,+,- ]$,
$S(46/100)=[-,-,+,+,- ]$ and $S(+\infty)=[+,+,+,-,-]$. Hence
$V(0)=3$, $V(45/100)=3$, $V(46/100)=2$ and $V(+\infty)=1.$ By Sturm
Theorem we deduce that $p$ has $V(0)-V(45/100)=0$ real roots in
$(0,45/100)$. Moreover $p$ has exactly $V(0)-V(+\infty)=2$ positive
real roots and the smallest one is in $(45/100,46/100).$ Hence,
since $p(0)>0$, it holds that $g(x)>p(x)>0$ on $[0,45/100]$,  as we
wanted to prove.

Notice that the proofs of Steps 1 and 2 in Theorem~\ref{main2}
follow the same ideas but involve much more computations which, for
the sake of simplicity, are  omitted. In that proofs  we used Maple
software to do the computations.

\section*{Acknowledgments}
The first author was   partially
 supported by grants MINECO/FEDER reference MTM2008-03437
 and Generalitat de Catalunya reference 2009-SGR410.  The second author by grants
  MINECO/FEDER reference
   MTM2009-08869 and MINECO reference  MTM2012-33937


\begin{thebibliography}{10}

\bibitem{Amos73}{\sc Amos, D. E.}, {Bounds of iterated coerror functions
 and their ratios}.
{\it Mathematics and Computation}  {\bf 27}  (1973) 413--427.

\bibitem{Avram13}{\sc Avram, F.}, On Dümbgen's exponentially modified Laplace continued fraction
for Mill's ratio. ArXiv: 1306.2989v1 (2013).

\bibitem{bar}
{\sc Baricz, \'{A}.}, Mills' ratio: monotonicity patterns and
functional inequalities. {\it J. Math. Anal. Appl.}
  {\bf 340} (2008) 1362--1370.

\bibitem{Bir42}  {\sc Birnbaum, Z. W.},  An inequality for Mill's ratio.
{\it  Ann. Math. Statistics} {\bf 13}  (1942) 245--246.

 \bibitem{Bir50} {\sc Birnbaum, Z. W.},  Effect of linear truncation on a multinormal population.
{\it Ann. Math. Statistics} {\bf 21}  (1950) 272--279.

\bibitem{Boyd59}{\sc Boyd, A. V.}, Inequalities for Mills' ratio. {\it Rep. Statst. Appl.
Res. Un. Jap. Sci. Engrs.} {\bf 6} (1959) 44--46.

\bibitem{Bryc02}{\sc Bryc, W.}, A uniform approximation to he right normal integral.
{\it Applied Mathematics and Computation} {\bf 127} (2002) 365-374.

\bibitem{Chang11}{\sc Chang, S.--H., Cosman, P. C. and Milstein L. B.},
{Chernoff-type bounds for the Gaussian error function}.
{\it IEEE Transactions on Communications}, {\bf 59}  (2011) 2939--2944.


\bibitem{C} {\sc Conrad, B.}, {\it Impossibility Theorems for Elementary Integration}.
 Academy Colloquium Series. Clay Mathematics Institute, Cambridge, MA. 2005.

 \bibitem{Cote12}{\sc Côté, F. D., Psaromiligkos, I. N. and Gross, W. J.},
 A Chernoff--type lower bound for the Gaussian {\it Q}--function.
 Arxiv:1202.6483v2 (2012).

\bibitem{cuyt}{\sc Cuyt, A.,  Petersen, V. B.,  Verdonk, B., Waadeland, H. and
  Jones, W. B.}, {\it  Handbook of Continued Fractions for Special Functions}.
Springer, New York,  2008.


\bibitem{Dum10} {\sc Dümbgen, L.}, Bounding standard Gaussian tail probabilities.
ArXiv:1012:2063v3 (2010).

\bibitem{fellerVol1}{\sc Feller, W.},
 {\it An Introduction to Probability Theory and Its Applications, Volume I},
Third Edition. Wiley, New York, 1968.

\bibitem{Hall79}{\sc Hall, P.},  On the rate of convergence
of normal extremes.
{\it J. Appl. Probab.}  {\bf 16}  (1979) 433--439.

\bibitem{Isaac94}{\sc Isaacson, E. and Keller, H. B.}, {\it
 Analysis of Numerical Methods}.
Dover, New York, 1994.


\bibitem{ItoMc74}{\sc Itô, K. and   McKean, H. P., Jr.},
{\it  Diffusion Processes and Their Sample Paths},
Second printing, corrected.
Springer, Berlin,  1974.

\bibitem{KaraLio07}{\sc Karagiannidis, G. K.  and Lioumpas, A. S.},
{ An improved approximation for the Gaussian {\it Q}--function}.
{\it IEEE Communication Letters} {\bf 11}  (2007)
644--646.

\bibitem{K} {\sc Kouba, O.},  Inequalities related to the error
function.   arXiv:math/0607694 (2006)

\bibitem{l} {\sc Liouville J.},
Memoire sur l'integration d'une classe de fonctions transcendantes.
{\it J. Reine Angew: Math.} {\bf 13} (1835), 93--118.

\bibitem{lag} {\sc  Laguerre E.},
Sur la r\'{e}duction en fractions continues d'une fraction qui satisfait
\`{a} une \'{e}quation diff\'{e}rentielle lin\'{e}aire du premier ordre dont les
coefficients sont rationnels. {\it  J. Math. Pures Appl. (4) } {\bf
1} (1885), 135--166.

\bibitem{laplace}{\sc Laplace, P. S.}, {\it Trait\'{e} de M\'{e}canique
C\'{e}leste, Tome IV}. Chez Courcier, Paris, 1805.



\bibitem{Mills}{\sc Mills, J. P.}, Table of the ratio: Area to bounding ordinate,
for any portion of normal curve. {\it Biometrika}, {\bf 18} (1926) 395-400.

\bibitem{Mitri70}{\sc Mitrinovi\'c}, D. S., {\it Analitic Inequalities}. Springer, Berlin 1970.

\bibitem{NIST10}
{\sc Olver, F. W. J., Lozier, D. W., Boisvert, R. F. and Clark, C. W.},
{\it  The NIST Handbook of Mathematical Functions}. Cambridge University Press, Cambridge, 2010.


\bibitem{p}  {\sc Pinelis, I.}, Monotonicity properties of the relative error of a Pad\'{e} approximation
for Mills' ratio. {\it J. Inequal. Pure Appl. Math.} {\bf 3}
(2002),  Article 20, 8 pp.

\bibitem{Poll56}{\sc Pollak, H. O.}, {A remark on ``Elementary Inequallities for Mills'
Ratio'' by Yûsaku Komatu}. {\it Rep. Statist. Appl. Res. Un. Jap. Sci. Engrs.}
{\bf 4} (1956)  40.


\bibitem{RayPit63}{\sc Ray, W. D. and Pitman, A. E. N. T.}, {Chebyshev polynomials
and other new approximations to Mills' ratio}. {\it Ann. Math. Statistics}
{\bf 34} (3) (1963) 892--902.

\bibitem{r}  {\sc Rosenlicht, M.},
Integration in finite terms. {\it Amer. Math. Monthly} {\bf 79} (1972)
963--972.

\bibitem{Sam}  {\sc Sampford, M. R.},
Some inequalities on Mill's ratio and related functions. {\it Ann.
Math. Statistics} {\bf 24} (1953) 130--132.

\bibitem{schilling12}{\sc Schilling, R., Song R. and  Vondra\v{c}ek, Z.},
{\it Bernstein Functions. Theory
and Applications}, 2nd Edition. De Gruyter, Berlin, 2012.

\bibitem{se}  {\sc Shenton, L. R.},
Inequalities for the normal integral including a new continued
fraction. {\it Biometrika} {\bf 41} (1954) 177--189.


\bibitem{small10}{\sc Small, C. G.}, {\it Expansions and Asymptotics for Statistics}.
 Chapman and Hall/CRC Press, Boca Raton, 2010.


\bibitem{s} {\sc  Sturm, J. C. F.},
M\'{e}moire sur la r\'{e}solution des \'{e}quations num\'{e}riques. {\it Bulletin des
Sciences de F\'{e}russac} {\bf 11} (1829), 419--425.

\end{thebibliography}
\end{document}